\documentclass[12pt]{amsart}

\usepackage[margin=2.4cm, headsep=0.75cm, footskip=0.75cm]{geometry}
\usepackage{amsmath, amsthm, amssymb, amsfonts, mathtools,mathrsfs}
\usepackage{color, xcolor}
\usepackage{enumitem}
\usepackage{csquotes}
\usepackage[mathscr]{euscript}
\usepackage{pgf, tikz, tikz-cd}
\usepackage[pdfusetitle]{hyperref}
\usepackage{aliascnt}
\usepackage{doi}
\usepackage[english]{babel}

\usepackage{color}

\usepackage[all]{xy}

\newcommand\go{G^{(0)}}

\newcommand{\C}{\mathbb{C}}
\DeclareMathOperator{\supp}{supp}
\DeclareMathOperator{\lsp}{span}

\newtheorem{thm}{Theorem}[section]

\newtheorem{lemma}[thm]{Lemma}
\newtheorem{prop}[thm]{Proposition}
\newtheorem{cor}[thm]{Corollary}

\theoremstyle{definition}
\newtheorem{definition}[thm]{Definition}

\theoremstyle{remark}

\newtheorem{remark}[thm]{Remark}

\newtheorem{example}[thm]{Example}

\numberwithin{equation}{section}

\begin{document}

\title{A Steinberg algebra approach to \'etale groupoid C*-algebras}

\author[L.O. Clark]{Lisa Orloff Clark}

\author[J. Zimmerman]{Joel Zimmerman}

\address[L.O.\ Clark]{School of Mathematics and Statistics, Victoria University of Wellington, PO Box 600, Wellington 6140, NEW ZEALAND}
\email{lisa.clark@vuw.ac.nz}
\address[J. \ Zimmerman]{School of Mathematics and Applied Statistics, University of Wollongong, Northfields Avenue, Wollongong 2522, AUSTRALIA}
\email{joelz@uow.edu.au}

\thanks{The authors thank Adam S\o rensen for suggesting we consider a \emph{counterfactual history} of groupoid C*-algebras by supposing Steinberg algebras came first.  
They also thank Astrid an Huef for many
helpful conversations. 
This research was supported by a Marsden grant from the Royal Society of New Zealand.  }

\begin{abstract} We construct the full and reduced C*-algebras of an ample groupoid from its complex Steinberg algebra.  We also show that our construction gives the same C*-algebras as the standard constructions.   In the last section, we consider an arbitrary locally compact, second-countable, \'etale groupoid, possibly non-Hausdorff.   Using the techniques developed for Steinberg algebras, we show that every $*$-homomorphism from Connes' space of functions to $B(\mathcal{H})$ is automatically I-norm bounded.  Previously, this was only known for Hausdorff $G$.  \end{abstract}

\subjclass [2010]{16S99, 16G30, 22A22, 46L05}

\keywords {Groupoid $C^*$-algebra, Steinberg algebra}

\maketitle

\section{Introduction}

C*-algebras of Hausdorff groupoids were introduced by Renault \cite{Ren80} as a completion of $C_c(G)$, the complex vector space of continuous functions from the groupoid $G$ to $\C$ that are zero outside of some compact set.  The construction for non-Hausdorff groupoids was developed (simultaneously, it seems!) by Connes in \cite{Connes80, Connes79}.  Instead of taking a completion of $C_c(G)$, Connes takes a completion of $\mathscr{C}(G)$, the complex vector space of functions from $G$ to $\C$ spanned by elements of $C_c(U)$ for all open Hausdorff subsets $U \subseteq G$.  (If $G$ is Hausdorff, $C_c(G)$ can be identified with $\mathscr{C}(G)$.)  In the full C*-algebra, the completion is taken with respect to the full norm, defined for $f \in \mathscr{C}(G)$ as is the supremum of the operator norms $\|\pi(f)\|$ taken over all bounded *-homomorphisms $\pi: \mathscr{C}(G) \to B(\mathscr{H})$ for some Hilbert space $\mathscr{H}$.

Exel gives an alternate approach  in \cite{Exe08} for \'etale groupoids, again by taking completions of $\mathscr{C}(G)$.   
He defines the full norm similarly; however, he takes the supremum over all *-homomorphisms, that is, he does not require them to be bounded.  
He demonstrates that when $G$ is \'etale,  the additional assumption is not needed to define a norm on $\mathscr{C}(G)$.  
This leaves open the question of whether or not Exel's  construction produces the same C*-algebra as Renault and Connes' constructions.

An \emph{ample} groupoid is an \'etale groupoid that has a basis of compact open sets.  The class of C*-algebras of ample groupoids contains many important subclasses including graph and higher-rank graph C*-algebras \cite{KP, KPRR}, Exel-Pardo algebras \cite{EP}, inverse semigroup algebras \cite{Pat99} and all Kirchberg algebras \cite{CFaH}.

In this paper, we build the full and reduced C*-algebra of an ample groupoid $G$ by taking completions of its Steinberg algebra $A(G)$, the complex vector space of functions from $G$ to $\C$ spanned by characteristic functions $1_B$ for all compact open Hausdorff subsets $B \subseteq G$,  (see \cite[Definition~4.1]{Ste10}).  As in \cite{BCFS} and \cite{CEPSS}, we advance the notion that the Steinberg algebra is a sort of Rosetta stone that can be used to make connections between a more general \'etale groupoid its C*-algebra.  

Given an ample groupoid $G$  with Hausdorff unit space, defining its reduced C*-algebra from $A(G)$ is straightforward;  moving from $\mathscr{C}(G)$ to $A(G)$ doesn't cause any problems.  To define the full C*-algebra,  we use Exel's approach from \cite{Exe08}.   The full norm of a function $f \in A(G)$ will be the supremum of $\|\pi(f)\|$ taken over all *-homomorphisms $\pi: A(G) \to B(\mathscr{H})$.  However, to show that this supremum is finite (without making a boundedness assumption on $\pi$), we require some innovation.  In Exel's original argument, for every compact open set $U \subseteq \go$, the algebra of continuous functions, $C(U)$ is a C*-algebra sitting inside $\mathscr{C}(G)$. Thus C*-algebraic tools are available. For us, $A(U)$ is not a C*-algebra (there is no obvious C$^*$-norm in which $A(U)$ is complete) so we have to work a little harder.  Once we establish that particular kinds of algebra *-homomorphism are norm decreasing  with respect to the uniform norm on $A(U)$, (see Proposition~\ref{prop:normdec}),  Exel's strategy for $\mathscr{C}(G)$ goes through.   We hope our construction will be a good starting point to learn about groupoid C*-algebras as there is less technical overhead.   Our approach also unifies the Hausdorff and non-Hausdorff  construction since the definition of $A(G)$ is
the same regardless of the Hausdorff property.

To reconcile our construction with the standard ones, we start by showing in Proposition~\ref{prop:autobound} that any $*$-homomorphism $\pi$ from $A(G)$ 
to $B(\mathscr{H})$ is automatically bounded with respect to the inductive limit topology (see definition~\ref{def:ILT}).  Then, for second countable $G$,  Renault's disintegration theorem \cite[Theorem~7.8]{MW08} implies $\pi$ is automatically I-norm bounded.   This proof depends on the structure of $A(G)$ and uses the basis of compact open sets in $G$.   But the techniques developed in the proof are our Rosetta stone.  By seeing how things played out in $A(G)$, we can translate to $\mathscr{C}(G)$ for more general  \'etale groupoids
using approximation arguments.   In particular, we show in Corollary~\ref{cor:etaleautobound}:  if $G$ is a second countable \'etale groupoid and $\pi$ is a  $*$-homomorphism from $\mathscr{C}(G)$ 
to $B(\mathscr{H})$, then $\pi$ is I-norm bounded.  Previously, this was an open question for non-Hausdorff groupoids.  

After a brief preliminaries section,  we give a self contained construction of the reduced and full  groupoid C*-algebra as the completion of $A(G)$ in Sections~\ref{sec:reduced} and \ref{sec:full}.   Note that our construction does not require $G$ to be second countable. 
 In section~\ref{sec:autobound} we show that $*$-homomorphisms from $A(G)$ to $B(\mathcal{H})$ are automatically bounded and in Section~\ref{sec:etaleautobound} we consider second-countable  \'etale groupoinds that are not necessarily ample.  In the last section, we show that the C*-algebras we have constructed are the same as the standard ones  (see Theorem~\ref{thm:recovery}).

\section{Preliminaries}
First some terminology.  When we say a subset $X$ of a topological space is \textbf{compact},  we mean that every open cover has a finite subcover, even for non-Hausdorff $X$.  We say a topological space  is \textbf{locally compact} if every point has a compact neighborhood base, again, without making any Hausdorff assumption.  The first author has also called this `locally locally compact', see \cite{CaHR} for further discussion.  

A groupoid is a generalisation of a group where the binary operation is only partially defined.  See \cite{Ren80} for a precise definition.   Let $G$ be a groupoid and $\go$ be the set of units in $G$ so that for each $\gamma \in G$ we have the maps
\[s(\gamma)=\gamma^{-1}\gamma \quad \text{and} \quad r(\gamma) = \gamma\gamma^{-1}\]
from $G$ to $\go \subseteq G$.
An \emph{\'etale groupoid} is a topological groupoid such that $s$ (and hence $r$) is a local homeomorphism.  In an \'etale groupoid, $\go$ is open in $G$; $\go$ is closed in $G$ if and only if $G$ is Hausdorff (see, for example, \cite[Proposition~3.10]{EP16}). 
We call an open set 
$B \subseteq G$ an \emph{open bisection} if $r$ and $s$ restricted to $B$ are homeomorphisms onto open subsets of $\go$. 
It is straightforward to check that a topological groupoid $G$ is \'etale groupoid if and only if $G$ has a basis of open bisections.  

We say $G$ is an \emph{ample groupoid} if $G$ has a basis of compact open bisections.  Equivalently, $G$ is ample if $G$ is locally compact and \'etale and the unit space $\go$ is totally disconnected \cite[Proposition~4.1]{Exe10}.  The collection of all compact open bisections in an ample groupoid $G$ forms an inverse semigroup with product and inverse 
\[BD:= \{bd : b\in B, d \in D, s(b)=r(d)\}\quad \text{and} \quad B^{-1} := \{b^{-1} : b \in B\}.\] 
(See \cite[Proposition~2.2.3]{Pat99}.)   
Note that, if $G$ is not Hausdorff, the intersection of compact open bisections need not be compact.

Because ample groupoids are necessarily \'etale, we have that the restriction of the range map to a compact open bisection $B$ induces a homeomorphism between $B$ and $r(B)\subseteq \go$.  So if $\go$ is Hausdorff,  then $B$ is also Hausdorff.  In this situation, techniques developed for Hausdorff groupoids can sometimes be employed locally.   For example, suppose that $B$ is a compact open bisection and that $\{D_i\}_{i \in I}$ is a finite cover of $B$ such that each $D_i \subseteq B$ is compact open. Then because $B$ is Hausdorff and each $D_i$ is compact, each $D_i$ is closed in the subspace $B$ and hence for any compact open $A \subseteq B$, 
we have that $D_i \setminus A$ is compact open.  So can find a disjoint cover of $B$ by compact open sets  $\{D_i'\}_{i \in I}$ where $D_1'=D_1$ and for $i>1$
 \[D_i' := D_i \setminus \bigcup_{j=1}^{i-1} D_j.\]
When we refine a cover in this manner, we say that we \emph{disjointify} the cover.

Let $G$ be an ample groupoid such that $\go$ is Hausdorff.   
For $B \subseteq G$, we write $1_B$ for the function from $G$ to $\mathbb{C}$ that takes value 1 on $B$ and 0 outside of $B$, that is, the characteristic function of $B$.
The complex Steinberg algebra of $G$ is the complex vector space 
\[A(G) := \lsp\{1_B \mid B \text{ is a compact open bisection}\}\]
where addition and scalar multiplication are defined pointwise.  
To make clear our notation,  for each $f \in A(G)$, we can write 
\[f=\sum_{B\in {F}} a_B1_B\] such that $F$ is a finite collection of compact open bisections and for each $B \in F$,  $a_B \in \mathbb{C}$.  
Then $A(G)$ is a $*$-algebra with convolution and involution of generators is given by 
\[1_B1_D= 1_{BD} \quad  \text{and} \quad (1_B)^* = 1_{B^{-1}}\] which distribute to give the usual convolution and involution formulae:
for $f,g \in A(G)$, and $\gamma \in G$
\[f*g(\gamma) = \sum_{\alpha\beta = \gamma} f(\alpha)g(\beta) \quad \text{and } \quad f(\gamma)^* = \overline{f(\gamma)} .\]
We will sometimes write convolution as $fg$ omitting the $*$.  
By \cite[Proposition~4.3 and Definition~4.1]{Ste10}, $A(G)$ is also equal to the span of $1_B$ ranging over all compact open Hausdorff subsets $B$ of $G$.

Since we are working with potentially non-Hausdorff groupoids, we need to take extra care with compact sets and closures. 
So for example, if $D$ is a compact open bisection, then $1_D$ is a continuous function on $D$. If $G$ is Hausdorff, the compact subsets of $G$ are closed and therefore $D$ is a clopen subset of $G$, which means that $1_D$ is a continuous on $G$. However, if $G$ is not Hausdorff then $D$ may not be closed in $G$ and then $1_D$  is not continuous on all of $G$.   See Example~\ref{ex:2snake}.

For $f \in A(G)$, we write
 \[\supp^o(f) = \{\gamma \in G : f(\gamma) \neq 0\}.\]
 Note that this set might not be open in $G$ and, although it is contained in a compact set (so the function is ``compactly supported''), its closure might fail to be compact.  See Example~\ref{ex:2snake}.  For more details on Steinberg algebras, see \cite{Ste10}. 

\begin{example}
\label{ex:2snake}
A basic non-Hausdorff example to keep in mind  is the ``two-headed snake'' groupoid:
the unit space $\go$ is the Cantor set (viewed as a subset of $[0,1] \subseteq \mathbb{R}$) and
$G = \go \cup \{\gamma_1\}$
where $s(\gamma_1)=r(\gamma_1) = 0$ and composition $\gamma_1\gamma_1 = 0$.  Equivalently, $G$ is a group bundle over the Cantor set where all the groups are trivial except for the one at 0, which is $\mathbb{Z}_2$.  To make $G$ into a topological groupoid, we take the usual topology on $\go$, that is, the subspace topology of $\mathbb{R}$.  In addition, for every open set $U$ containing 0, we add an open set $U_1$ where we remove 0 and add $\gamma_1$.  
Then $G$ is an ample groupoid that has a Hausdorff unit space but is not Hausdorff itself.  

The unit space is a compact open subset of $G$ so $1_{\go} \in A(G)$ but since $\go$ is not closed in $G$, $1_{\go}$ is not continuous.
Let $B = (\go \setminus \{0\}) \cup \{\gamma_1\}$ which is a compact open bisection in $G$.  
Then for $f = 1_{\go} - 1_B$, we have $\supp^o(f) = \{0, \gamma_1\}$ which is not open.   

Since $G$ is compact, for every $f \in A(G)$,   the closure $\overline{\supp^o(f)}$ is compact.  
However, we can add more heads to the snake to get an example where this fails:  let $H$ be 
the group bundle over the Cantor set where all the groups are trivial except for the one at 0, 
which is $\mathbb{Z}$.   Denote the non-unit elements of $H$ by $\gamma_i$ for  $i \in \mathbb{Z}\setminus \{0\}$.  For the topology, we have the usual topology on the Cantor set and for each $i \in \mathbb{Z}$ and 
each open set containing 0, we add an open set where we replace 0 with $\gamma_i$. Then $H$ is an ample groupoid with Hausdorff unit space.  
Now $1_{\go} \in A(G)$ but  $\overline{\supp^o(1_{\go})} = G$ is not compact.
\end{example}

We will need the following lemma in the sequel.  It is obvious when $G$ is Hausdorff but more generally it is surprisingly technical.
\begin{lemma}
\label{lem:supportsinC}
Let $G$ be an ample groupoid such that $\go$ is Hausdorff.   
Let $f \in A(G)$ with $\supp^o(f) \subseteq C$ where $C \subseteq G$ is a compact open subset.  
Suppose  $f  = \sum_{B\in {H}} a_B1_B$ where ${H}$ is a finite collection of compact open bisections.  
Then we can write 
$f  = \sum_{D\in F} a_D1_D$ where $F$ is a finite collection of compact open bisections such that for each $D \in F$ we have $D \subseteq C$.
\end{lemma}
 \begin{proof}
 We do a proof by induction on $|H|$, the size of $H$.  The base case $|H| =1$ is clear.  
 For the inductive step, suppose the lemma holds for functions $f \in A(G)$ with $|H| =n \geq 1$.  
Fix $f \in A(G)$ with $\supp^o(f) \subseteq C$ such that 
$f  = \sum_{B\in {H}} a_B1_B$ where $|H|=n+1$.  Suppose there exists $B_1 \in {H}$ such that $B_1 \setminus C \neq \emptyset$.  (If no such $B_1$ exists, we are done.)   
Notice that $B_1 \setminus C$ is closed in $B_1$ and hence is compact.  Fix $\gamma \in B_1 \setminus C$.  Since $f(\gamma)=0$, there exists a maximal $S_{\gamma} \subseteq {H}$ such that
\[B_1 \in S_{\gamma}, \ \ \gamma \in \bigcap_{B \in S_\gamma} B, \text{ and }\sum_{B \in S_{\gamma}}a_B = 0.\]  
Because $G$ is ample and elements of ${H}$ are open, we can find a compact open bisection $D_{\gamma}$ such that  
\[\gamma \in D_{\gamma}\subseteq \bigcap_{B \in S_{\gamma}}B.\]
 The collection $\{D_{\gamma}\}_{\gamma \in B_1 \setminus C}$ covers $B_1 \setminus C$ and hence has a finite subcover
$\{D_{\gamma}\}_{\gamma \in I}$ for some finite $I \subseteq B_1 \setminus C$.  
Since we are inside the Hausdorff space $B_1$, we can disjointify and assume this finite cover of $B_1 \setminus C$ is disjoint.  
Now define compact open bisections for each $B \in {H}$
\[
D_{B} := B \setminus \bigcup_{\{\gamma \in I : B \in S_{\gamma}\}} D_{\gamma}.
\]
Notice that for our fixed $B_1$ we have $D_{B_1} \subseteq C$.
Let $F_1 = \{D_{B} : B \in {H}\}$.
We claim that 
\[f  = \sum_{D_{B} \in {F_1}} a_{B}1_{D_{B}}.\]
Fix $\alpha \in G$.  
We consider two cases.  First suppose 
$\alpha \notin D_{\gamma}$ for all $\gamma \in I$.   
Then 
\[\alpha \in B \in F  \iff  \alpha \in D_{B} \in F_1\]
 so the claim is true.
Now suppose there exists $\gamma \in I$ such that $\alpha \in D_{\gamma} \subseteq \bigcap_{B  \in S_{\gamma}}B$.  
Since the collection is disjoint there is only one such $\gamma \in I$.
By construction $\sum_{B \in S_{\gamma}}a_{B} = 0$.
So 
\begin{align*}
f(\alpha) &= \sum_{\{B \in {H} : \alpha \in B\}} a_B \\
&=  \sum_{\{D_B \in F_1 : \alpha \in D_B\}} a_B +  \sum_{ B \in S_\gamma} a_B\\
&=  \sum_{\{D_B \in F_1 : \alpha \in D_B\}} a_B
\end{align*}
proving the claim.
Thus \[f =  \sum_{D_{B} \in {F_1}} a_{B}1_{D_{B}} = a_{B_1}1_{D_{B_1}} + \sum_{D_{B} \in {F_1}, B \neq B_1} a_{B}1_{D_{B}} \]
with $D_{B_1} \subseteq C$.
Now the inductive hypothesis applies to the second term, proving the lemma
\end{proof}


\section{The reduced C*-algebra}
\label{sec:reduced}

To define the reduced C*-algebra of an ample groupoid $G$ from $A(G)$, the usual approach for $\mathscr{C}(G)$ goes through  virtually unchanged so we just outline the process.  
The idea is to embed $A(G)$ into $B(\mathcal{H})$ for a specific $\mathcal{H}$ using an injective $*$-homomorphism $\pi$ and then define the reduced norm of a function in $A(G)$ to be the operator norm of its image under $\pi$. Then the reduced C*-algebra $C^*_r(G)$ is the completion of $A(G)$ with respect to this norm,  and is isomorphic to the closure of $\pi(A(G))$ in $B(\mathcal{H})$.

The \emph{$I$-norm} for  $f \in A(G)$ is defined as
\[ \|f\|_I := \sup_{x \in \go} \bigg\{\sum_{\{\gamma : s(\gamma)=x\}}|f(\gamma)|, \sum_{\{\gamma : r(\gamma)=x\}} |f(\gamma)|\bigg\}.\]
We begin by introducing for each $x \in G^{(0)}$, an I-norm bounded *-homomorphism
 \[
 \pi_x: A(G) \to {B}(\ell^2(G_x))
 \] which is called the \emph{regular representation of $A(G)$ at $x$}. 
Write the standard orthonormal basis of $\ell^2(G_x)$ as the indicator functions $\{\delta_{\gamma}\}_{\gamma \in G_x}.$
Then if $B$ is a compact open bisection in $G$ and $\gamma \in G_x$, we define $\pi_x$ on generators by the formula 
\[\pi_x(1_B)\delta_{\gamma} = \delta_{B\gamma}\]
where $B\gamma$ is shorthand for the product of sets $B\{\gamma\}$. Since $B$ is a bisection, $B\gamma$ is either empty or a singleton. 
More generally, we have the following.  The proof follows similarly to the proof of \cite[Proposition~3.3.1]{Sims18}.

\begin{prop}For each $x \in G^{(0)}$, there exists a $*$-homomorphism $\pi_x : A(G) \to {B}(\ell^2(G_x))$ defined by 
\[\pi_x(f)\delta_{\gamma} = \sum_{\alpha \in G_{r(\gamma)}} f(\alpha)\delta_{\alpha\gamma}\]
such that $\|\pi_x(f)\| \leq \|f\|_I$.
\end{prop}    
\noindent It is useful to observe that the formula for $\pi_x$ is the same as the convolution product formula in $A(G)$.  
Let $\mathcal{H} = \bigoplus_{x \in \go}\ell^2(G_x).$
Then there is an injective $*$-homomorphism $\pi:A(G) \to B(\mathcal{H})$ such that for 
$(g_x)_{x\in \go} \in \mathcal{H}$ and $f \in A(G)$ we have 
\[\pi(f)(g_x)_{x\in\go} = (\pi_x(f)g_x)_{x\in \go} \quad
\text{and} \quad \|\pi(f)\| = \sup_{x\in\go}\|\pi_x(f)\|.\]
We call $\pi$ the \emph{left regular representation of $A(G)$}. 
Now the reduced norm  for $f\in A(G)$ is defined such that  $\|f\|_r =\|\pi(f)\|$ and $C^*_r(G)$ is identified with $\overline{\pi(A(G))}$ in $B(\mathcal{H})$.

\section{The full C*-algebra}
\label{sec:full}

To construct the full C*-algebra, we show that for $f \in A(G)$,  we can define a norm $\|f\|$ with the formula
  \[ \sup \{\|\pi(f)\| \mid  \pi:A(G) \to B(\mathscr{H}) \text{ is a *-homomorphism for some Hilbert space } \mathscr{H}\}.\] 
Thus we must show that the set
 \[
P_f:=\{\|\pi(f)\| \mid  \pi:A(G) \to B(\mathscr{H})\text{ is a *-homomorphism for some Hilbert space } \mathscr{H}\}
\] has an upper bound. 
 We do this in stages:
\begin{itemize}
\item We first show $P_f$ has an upper bound when  $\supp^o(f) \subseteq \go$.  
\item Then we show $P_f$ has an upper bound when $\supp^o(f) \subseteq B$ for some compact open bisection $B$.
\item Finally we show $P_f$ has an upper bound for any $f \in A(G)$.
\end{itemize}

\noindent\textbf{Functions supported on $\go$:}  
If $U \subseteq \go$ is a compact open subset, then $U$ is itself a groupoid that consists entirely of units and $A(U)$ is a complex Steinberg algebra.  
In the next proposition we show 
that when $\pi:A(U) \to B(\mathscr{H})$ is a $*$-homomorphism, then $\pi$ is automatically norm-decreasing.

\begin{remark}In Exel's construction in \cite{Exe08}, instead of $A(U)$, he works with $C(U) \subseteq \mathscr{C}(G)$, which is a C*-algebra.  So he can use the fact that 
any homomorphism between C*-algebras is automatically norm-decreasing, see, for example, \cite[Theorem~2.1.7]{Murphy}.
\end{remark} 

\begin{prop}\label{prop:normdec}
Let $G$ be an ample groupoid with Hausdorff unit space and let $U \subseteq \go$ be compact open.  Suppose 
\[\pi:A(U) \to B(\mathscr{H})\]
is a *-homomorphism for some Hilbert space $\mathscr{H}$.  Then $\pi$ is norm-decreasing with respect to the uniform norm on $A(U)$ and the
operator norm on $B(\mathscr{H})$.  
\end{prop}

\begin{proof}
Notice that $A(U)$ is unital with identity $1_U$.  Let $B =\pi(A(U))$.  Then $\pi(1_U)$ is an identity in $B$.  By continuity of multiplication in $B(\mathscr{H})$, $\pi(1_U)$ is an identity in $\overline{B}$. 

Since $\overline{B}$ is a C*-algebra, each element $b \in \overline{B}$ has a spectrum $\sigma_{\overline{B}}(b)$ and spectral radius $r_{\overline{B}}(b)$. 
Fix $f \in A(U)$.  Define the algebraic spectrum of $f$ to be
\[\sigma_{A(U)}(f):=\{\lambda \in \sigma_{C(U)}(f) : f-\lambda1_U \text{ is not invertible in }  A(U)\}\]
where $\sigma_{C(U)}(f)$ denotes the usual spectrum of $f$ in the C*-algebra $C(U)$ of continuous functions on $U$.  Similarly define the spectral radius
\[
r_{A(U)}(f):= \sup\{|\lambda| : \lambda \in \sigma_{A(U)}(f)\}.
\]
In fact  $\sigma_{A(U)}(f) = \sigma_{C(U)}(f)$ and $r_{A(U)}(f) = r_{C(U)}(f)$, we are just being careful to observe which algebra we are working in.

 We claim that $r_{\overline{B}}(\pi(f)) \leq r_{A(U)}(f)$.  
Suppose $\lambda \notin \sigma_{A(U)}(f)$.  Then $(f-\lambda 1_U)^{-1} \in A(U)$.  
Since $\pi$ is a unital homomorphism from $A(U)$ into $\overline{B}$,  
$\pi((f-\lambda 1_U)^{-1}) \in \overline{B}$ is an inverse for $\pi(f-\lambda 1_U)$ and so by linearity,
$\lambda \notin \sigma_{\overline{B}}(\pi(f))$.  Thus
$\sigma_{\overline{B}}(\pi(f)) \subseteq \sigma_{A(U)}(f)$ and the claim follows.
Now because $\overline B$ is a C*-algebra we use \cite[Theorem~2.1.1]{Murphy} to compute
$$\|\pi(f)\|^2 = \|\pi(f)\|^2_{\overline{B}} =\|\pi(f^*f)\|_{\overline{B}}=r_{\overline{B}}(\pi(f^*f)), $$ 
further, we have $$ r_{A(U)}(f^*f) =r_{C(U)}(f^*f)=\|f\|_{\infty}^2,$$ 
and by our claim and the above calculations, we observe that $\|\pi(f)\|^2 \leq \|f\|^2_{\infty}$ as required.
\end{proof}

For compact open $U \subseteq \go$, we identify $A(U)$ with its isomorphic image in $A(G)$ under the inclusion map 
$i:A(U) \to A(G)$ where
\[  i(f)(\gamma) = \begin{cases}f(\gamma) & \text{if } \gamma \in U\\  0 & \text{otherwise}. \end{cases}\]
The next lemma shows that our intuition about ``functions supported on the unit space'' is accurate. 

\begin{lemma}
\label{lem:gap}
Let $G$ be an ample groupoid such that $\go$ is Hausdorff.   
If $f\in A(G)$ with $\supp^o(f) \subseteq \go$, then there is a compact open subset $U \subseteq \go$ such that
$\supp^o(f) \subseteq U$.
\end{lemma}

\begin{proof}
Write $f = \sum_{B \in F} a_B1_B$ where $F$ is a finite collection of compact open bisections.  Then 
\[
\supp^o(f) \subseteq K := \bigcup_{B \in F} B.
\]
Notice that $K$ is compact open.  
We have $\supp^o(f) \subseteq \go$ by assumption, so $s(\supp^o(f)) = \supp^o(f)$.
Now 
\[\supp^o(f) \subseteq s(K)\]
which is compact open because $s$ is a local homeomorphism. So $U = s(K)$ suffices.  
\end{proof}

We can now verify that $P_f$ is bounded by $\|f\|_{\infty}$ for and function $f$
 supported on $\go$.

 \begin{lemma}
 \label{lem:supuni} Let $\pi:A(G) \to B(\mathscr{H})$ for some Hilbert space $\mathcal{H}$.
 If $f\in A(G)$ with $\supp^o(f) \subseteq \go$,
 then $\|\pi(f)\| \leq \|f\|_{\infty}$.
 \end{lemma}

\begin{proof}
We have $\supp^o(f) \subseteq U$ for some compact open $U \subseteq \go$  by Lemma~\eqref{lem:gap}.  
Observe that $\pi|_{A(U)}$ is a *-homomorphism from $A(U)$ to $B(\mathscr{H})$.
Then $\|\pi(f)\| = \|\pi|_{A(U)}(f)\| \leq \|f\|_{\infty}$ by Proposition~\ref{prop:normdec}.
\end{proof}

 \noindent\textbf{Functions supported on a bisection:}   If $f \in A(G)$ with $\supp^o(f) \subseteq B$ for some compact open bisection $B$, then once again we get that $P_f$ has upper bound $\|f\|_{\infty}$.

 \begin{lemma}
 \label{lem:supbis} Let $\pi:A(G) \to B(\mathscr{H})$ for some Hilbert space $\mathcal{H}$.
 If $f\in A(G)$ with $\supp^o(f) \subseteq B$ for some compact open bisection $B$,
 then $\|\pi(f)\| \leq \|f\|_{\infty}$.
 \end{lemma}

\begin{proof}
Fix $f \in A(G)$ such that $\supp^o(f)$ is contained in a compact open bisection.  Then $\supp^o(f^*f) \subseteq \go$ by \cite[Proposition~3.12]{Exe08}.   Now Lemma~\ref{lem:supuni} gives the result:
\[
\|\pi(f)\|^2 =\|\pi(f^*f)\| \leq \|f^*f\|_{\infty} = \|f\|^2_{\infty}.\]
\end{proof}

 \noindent\textbf{An arbitrary element of $A(G)$:} For an arbitrary $f \in A(G)$,  $\|f\|_{\infty}$ is no longer an upper bound for $P_f$;  however, since $f$ is a sum of things that are each bounded, we can still find a bound for $P_f$.

  \begin{lemma}\label{lem:Pfisbounded}
 \label{lem:anyf} Let $\pi:A(G) \to B(\mathscr{H})$ for some Hilbert space $\mathcal{H}$.
 If $f\in A(G)$ then there exists $M_f \geq0$ such that $\|\pi(f)\| \leq M_f$.
 \end{lemma}

\begin{proof}
Write $f = \sum_{D \in F} a_D1_D$ where each $D \in F$ is a compact open bisection and each $a_D \in \C$.  
Now we use that $\pi$ is a homomorphism and apply Lemma~\ref{lem:supbis} to get 
\[\|\pi(f)\| = \|\pi( \sum_{D \in F} a_D1_D)\| \leq  \sum_{D \in F} \|\pi(a_D1_D)\| \leq  \sum_{D \in F} |a_D|.\]
So $M_f = \sum_{D \in F} |a_D|$ suffices. 
\end{proof}

\begin{thm}
\label{thm:fullnorm}
Let $G$ be an ample groupoid with Hausdorff unit space.
For $f\in A(G)$, the formula  
\[\|f\| = \sup \{\|\pi(f)\| \mid \pi:A(G) \to B(\mathscr{H})\text{ is a *-homomorphism for some Hilbert space } \mathscr{H}\}\] is a C*-norm on $A(G)$. 
\end{thm}

\begin{proof}
That the supremum exists follows from Lemma~\ref{lem:Pfisbounded}.
It is straightforward to check that this is a C*-seminorm by noting that the maps $f \mapsto \|\pi(f)\|$ are all C*-norms in their own right, and by noting that properties such as the triangle inequality, submultiplicativity, and the C*-identity are all preserved after taking the supremum.  If $f$ is nonzero, then $\|f\|$ is nonzero since the left regular representation is injective.  
\end{proof}

\begin{definition}
\label{def:main}
Let $G$ be an ample groupoid with Hausdorff unit space.
Define $C^*(G)$ to be the completion of $A(G)$ in the norm $\| \cdot \|$ from Theorem~\ref{thm:fullnorm}.
\end{definition}


\section{$*$-homomorphisms are automatically bounded}
\label{sec:autobound}

\noindent{\textbf{The full C*-algebra}:}  In order to show our full C*-algebra is the same as the standard one, we establish Proposition~\ref{prop:autobound}.  We could have used this proposition to avoid all of the work above to get a bound on $\|\pi(f)\|$ for any $*$-homomorphism $\pi:A(G) \to B(\mathscr{H})$.  However, the above construction is much less technical and gives a simplified, self-contained construction of $C^*(G)$ which was one of our goals.  

Given a topological space $U$, recall that $C_c(U)$ is the set of continuous functions from $U$ to $\mathbb{C}$ that are zero outside of some compact set in $U$.
So if $U$ is a Hausdorff subspace of some bigger space $X$ and $f \in C_c(U)$, then
$\overline{\supp^o(f)}^U$, where  closure is with respect to the subspace topology on $U$,  is compact.    
Recall that
\[\mathscr{C}(G) = \operatorname{span} \{f  \in C_c(U) \mid U \subseteq G \text{ is open and Hausdorff}\}.\]    
 We view each $C_c(U)$ as sitting inside the vector space of all functions from $G$ to $\mathbb{C}$ by defining them to be zero outside of $U$ and hence $A(G) \subseteq \mathscr{C}(G)$.  

\begin{definition}
\label{def:ILT}
Following \cite{MW08}: \begin{itemize}
\item We say a net of functions $\{f_n\}$ \emph{converges to $f$ in the inductive limit topology} if $f_n \to f$ uniformly and $\supp^o(f_n) \subseteq K$ eventually for some compact $K$.  Like Muhly and Williams,  we are not claiming there is a topology where these are the only convergent sequences.    
\item We will also say a map  with domain $A(G)$ (or $\mathscr{C}(G)$) is continuous with respect to the inductive limit topology if it takes convergent nets in the inductive limit topology to convergent nets in the codomain. 
\end{itemize}
\end{definition}

\begin{prop}
\label{prop:autobound}
Let $G$ be an ample groupoid with Hausdorff unit space.
Suppose $\pi:A(G) \to B(\mathscr{H})$ is a *-homomorphism for some Hilbert space $\mathscr{H}$.  Then $\pi$ is continuous with respect to the inductive limit topology on $A(G)$.
\end{prop}

Before presenting the proof,  we need to make sure functions in $A(G)$ can be written in a somewhat efficient way.  When $G$ is Hausdorff, this is easy because functions in $A(G)$ are locally constant and have finite range.  

\begin{lemma}
\label{lem:rest}
Let $B,D$ be compact open bisections with $B \subseteq D$.  If $f \in A(G)$ with $\supp^o(f) \subseteq D$, 
then $f|_{B} \in A(G)$.
\end{lemma}

\begin{proof}
Using  Lemma~\ref{lem:supportsinC}, write $f = \sum_{C \in F} a_C 1_C$ where each $C \in F$ is a compact open bisection in $D$.    Since $D$ is an open bisection, it is Hausdorff thus for each $C \in F$, $B \cap C$ is also a compact open bisection.  Now 
\[  f|_{B}  = ( \sum_{C \in F} a_C 1_C)|_{B} =  \sum_{C \in F} a_C 1_{C\cap B} \in A(G).\]
\end{proof}

\begin{lemma}
\label{lem:boundedsummands}
Let $K\subseteq G$ be compact and $f \in A(G)$ with $\supp^o(f) \subseteq K$.  Suppose  $B_1, ..., B_k$ are compact open bisections that cover $K$. 
Suppose $\epsilon >0$. Then for  each $i, 1 \leq i \leq k$, there exists $f_{i}\in A(G)$ with  $\supp^o(f_{i}) \subseteq B_i$ such that 
\begin{enumerate}
\item \label{it1:boundedsummand}we have $\|f_{i}\|_{\infty}\leq  \|f\|_{\infty}+\epsilon$ and 
\item\label{it2:boundedsummand}$f = \sum_{i=1}^k f_{i}$.
\end{enumerate}
\end{lemma}

\begin{proof}
Fix $\epsilon > 0$. 
Write $f=\sum_{D \in F}f_D$ where $F$ is a finite collection of compact open bisections and for each $D \in F$,  $f_D = a_D1_D$ with nonzero $a_D \in \C$.  
We can assume each $D \subseteq \bigcup_{I=1}^k B_i$ by Lemma~\eqref{lem:supportsinC}. 
Fix $D \in F$.  Since $D \subseteq K$, the collection $\{D \cap B_i :   1 \leq i \leq k\}$ is a cover of $D$ by open sets.   Since $G$ is ample and hence as a basis of compact open bisections, for each $\gamma \in D$, there exist $1 \leq i_\gamma \leq k$ and  a compact open bisection $D_{\gamma,i_{\gamma}}$ such that
$\gamma \in D_{\gamma,i_{\gamma}} \subseteq D \cap B_{i_{\gamma}}$.  The collection $\{D_{\gamma,i_{\gamma}}\}_{\gamma \in D}$ covers $D$ and hence has a finite subcover.  Since each $D_{\gamma,i_{\gamma}}$ is compact open in the the Hausdorff subspace $D$, it is clopen in $D$.   So we can disjointify (and relabel) to get a finite disjoint cover of $D$ by compact open bisections $\{D_{\gamma,i_\gamma}\}_{\gamma \in I_D}$ for a finite set $I_D \subseteq D$ where each $D_{\gamma,i_{\gamma}} \subseteq D \cap B_{i_{\gamma}}$.  So 
\[f_D = \sum_{\gamma \in I_D} (f_D)|_{D_{\gamma,i_{\gamma}}}\]
 and each summand is in $A(G)$ by Lemma~\ref{lem:rest}. 

 For pairs $(\gamma,i)$ with $\gamma \in I_D$ and $i \neq i_{\gamma}$,  define  $D_{\gamma,i}$ to be the empty set and write
\begin{equation}\label{eq:f}
f=  \sum_{D\in F}\left(  \sum_{\gamma \in I_D}  (f_D)|_{D_{\gamma,i_{\gamma}}} \right) = \sum_{D \in F} \left( \sum_{i=1}^k  \sum_{\gamma \in I_D} (f_D)|_{D_{\gamma,i}} \right)
= \sum_{i=1}^k \left( \sum_{D\in F}  \sum_{\gamma \in I_D} (f_D)|_{D_{\gamma,i}} \right).
\end{equation}
We start out by defining ${f}_{i}^0 \in C_c(B_i) \cap A(G)$ to be the function in the $i$th summand in the last expression for $f$ in \eqref{eq:f}. 
Then \eqref{it2:boundedsummand} holds for the $f_{i}^0$'s.   Since each  ${f}_{i}^0$ is defined on an open Hausdorff subset, it is locally constant.  

Here things get more technical. In what follows, we adjust each of these functions $f_i^0$ so that  \eqref{it2:boundedsummand} still holds 
in order to ensure  \eqref{it1:boundedsummand} holds as well.   The argument that follows does not depend on the exact definition of $f_{i}^0$, it only uses that each $f_{i}^0$ is a locally constant element of $A(G)$ that is zero outside $B_i$ and that \eqref{it2:boundedsummand} holds.

Fix $\epsilon>0$.   If 
\begin{equation}
\label{eq:prop}
 \|f_{i}^0\|_{\infty}\leq  \|f\|_{\infty}+\epsilon
\end{equation}
is true for all $i$, then for each $i$, set $f_i =f_i^0$ and we are done.  
Otherwise, fix $i$ such that \eqref{eq:prop} fails for $i$.  Thus $ \|f_i^0\|_{\infty} > \|f\|_{\infty} +\epsilon$.
Consider the set
\[O = \{\gamma \in B_i : |f^0_{i}(\gamma)| < \|f\|_{\infty} + \epsilon\},\]

Then  $O$ is open as it is the inverse image of an open set under the composition of two continuous maps.
Since  $\|f^0_{i}\|_{\infty} > \|f\|_{\infty}+\epsilon$, the set $B_i \setminus O$ is nonempty. 
Further $B_i\setminus O$ is closed in $B_i$ and hence compact.  

Fix $\gamma \in B_i \setminus O$.  
Then $|f^0_{i}(\gamma)| > \|f\|_{\infty}$ so in particular $f(\gamma) \neq f^0_{i}(\gamma)$ and hence there exists a maximal nonempty 
\[S_{\gamma}\subseteq \{1, ..., i-1, i+1, ..., k\}\] 
such that 
\[\gamma \in B_i \cap \left(\bigcap_{j \in S_{\gamma}}B_j\right)\]  and 
\[f(\gamma) = f_{i}^0(\gamma) + \sum_{j \in S_{\gamma}} f_{j}^0(\gamma).\]

Choose a compact open bisection 
$C_{\gamma, S_{\gamma}} \subseteq B_i \cap \left( \bigcap_{j \in S_{\gamma}}B_j \right)$ such that $\gamma \in C_{\gamma,S_{\gamma}}$ and for each $j \in S_{\gamma}$,  ${f}_{j}^0$ is constant on $C_{\gamma, S_{\gamma}}$.  
Then the collection $\{C_{\gamma, S_\gamma}\}_{\gamma \in B_i \setminus O}$ covers $B_i \setminus O$ so there exists a finite subcover.  
Since this is all  taking place inside of the Hausdorff subspace $B_i$, we can disjointify to get a finite disjoint  subcover 
$\{C_{\gamma_p,S_p}\}_{p=1}^s$ of $B_i \setminus O$ such that for any $\alpha \in C_{\gamma_p,S_{p}}$ we have
\begin{equation}
\label{eq:gammap}
{f}_{i}^0(\alpha) + \sum_{j \in S_{p}} {f}_{j}^0(\alpha) = f(\gamma_p).\end{equation}

Now, we adjust the functions.  First, we define $f_{i}$ as follows:
\[
f_{i}={f}_{i}^0 + \sum_{j=1}^k\  \sum_{\{p: j \in S_p\}} ({f}_{j}^0)|_{C_{\gamma_p,S_p}}.\]
To ensure \eqref{it2:boundedsummand} still holds,  for $j \neq i$ define 
 \[
f_{j} = {f}_{j}^0 - \sum_{\{p: j \in S_p\}} ({f}_{j}^0)|_{C_{p,S_p}}.
\]  In each case, we stay inside of $A(G)$ by Lemma~\ref{lem:rest}.

Notice for $j \neq i$ we have  $\|f_{j} \|_{\infty} \leq \| {f}_{j}^0\|_{\infty}$ as we are just forcing it to be 0 on a larger set.
We claim that $\|f_{i}\|_{\infty} \leq \|f\|_{\infty}+\epsilon$.  
To prove the claim, fix $\alpha \in G$.  It suffices to show $|f_{i}(\alpha)| \leq \|f\|_{\infty}+\epsilon$.   
If for every $p$,  $\alpha \notin C_{\gamma_p,S_p}$, then $\alpha \in O$ and  
\[
|f_{i}(\alpha)|=|f_{i}^0(\alpha)| < \|f\|_{\infty}+\epsilon.  
\]
On the other hand, if there exists $q$ such that $\alpha \in C_{\gamma_q, S_q}$, then $q$ is unique since the collection is disjoint.  Now
\begin{align*}
f_{i}(\alpha) &= {f}_{i}^0(\alpha) + \sum_{j=1}^k \ \sum_{\{p: j \in S_p\}} ({f}_{j}^0)|_{C_{p,S_p}} (\alpha) \\
&= {f}_{i}^0(\alpha) +  \sum_{j \in S_q} ({f}_{j}^0) (\alpha) \\
&= f(\gamma_p)
\end{align*}
 by \eqref{eq:gammap}.  Thus $|f_{i}(\alpha)| \leq \|f\|_{\infty}$ proving the claim.  
Thus 
\[\|f_{i}\|_{\infty} < \|f\|_{\infty} + \epsilon.\]

Now check if there are any $j$ such that  \eqref{eq:prop} fails.  If so, pick one such $j$ and repeat the above process to get a new collection of functions continuing until no such $j$ exists and hence \eqref{it1:boundedsummand} holds. 
\end{proof}

We present one last lemma before presenting the proof of Proposition~\ref{prop:autobound}.  It is well-known but we provide the details as they don't seem to appear explicitly in the literature.  It clarifies that we are in a situation where we can apply the disintegration theorem  \cite[Theorem~7.8]{MW08} in Corollary~\ref{cor:autobound}.  

\begin{lemma}
\label{lem:ILTconv}
Let $G$ be an ample groupoid with $\go$ Hausdorff.  Then $A(G)$ is dense in $\mathscr{C}(G)$ with respect to the inductive limit topology in that for every $f$ in $\mathscr{C}(G)$, there exists a net $(f_n) \subseteq A(G)$ such that $f_n \to f$ in the inductive limit topology. 
\end{lemma}

\begin{proof}
Fix $f \in \mathscr{C}(G)$.  Then $f = \sum_{i=1}^n f_i$ where each $f_i \in C_c(B_i)$ for compact open bisections $B_i$ by \cite[Proposition~3.10]{Exe08}.
For each $i$, $B_i$ is a compact Hausdorff space and hence $f_i$ is the uniform limit of a net functions 
$(f_{i,n}) \subseteq A(G) \cap C_c(B_i)$ by the Stone Weierstrauss theorem.   
Thus $\sum_{i}f_{i,n}$ converges uniformly to $f$ inside the compact set $\bigcup_{i=1}^n B_i$ and the result follows. 
\end{proof}

\begin{proof}[Proof of Proposition~\ref{prop:autobound}]
Fix a $*$-homomorphism $\pi:A(G) \to B(\mathscr{H})$. 
Suppose we have a net $(f_n)_{n \in (J,\preceq)}$  in $A(G)$ such that $f_n \to f \in A(G)$  uniformly and $\supp^o(f_n) \subseteq K$ eventually for some compact $K$.  
Since $K$ is compact, there exists a finite collection of compact open bisections  $B_1, ..., B_k$ that cover $K$.  
Fix $\epsilon >0$ and apply Lemma~\ref{lem:boundedsummands} to write each $f_n-f = \sum_{i=1}^k f_{n,i}$  such that each $f_{n,i} \in A(G)$ with  $\supp^o(f_{n,i}) \subseteq B_i$ and   
\[\|f_{n,i}\|_{\infty}\leq \|f_n-f \|_{\infty}+ \dfrac{\epsilon}{2k}.\]  
Choose $\alpha$ so that 
\[ \|f_n-f\|_{\infty} \leq \dfrac{\epsilon}{2k}\]
for $\alpha \preceq n$.  
Now we compute, applying Lemma~\ref{lem:supbis} at the second inequality 
\[\|\pi(f_n-f)\|\leq \sum_{i=1}^k\|\pi( f_{n,i})\| \leq \sum_{i=1}^k \|f_{n,i}\|_{\infty} \leq k (\|f_n-f\|_{\infty}+\dfrac{\epsilon}{2k}) \leq \epsilon  \]
when $\alpha \preceq n$.

\end{proof}

\begin{cor}
\label{cor:autobound}
Let $G$ be a second countable ample groupoid with Hausdorff unit space.
Suppose $\pi:A(G) \to B(\mathscr{H})$ is a *-homomorphism for some Hilbert space $\mathscr{H}$.  Then $\pi$ is bounded with respect to the $I$-norm on $A(G)$.
\end{cor}
\begin{proof}
From Proposition~\ref{prop:autobound} we have that $\pi$ is continuous with respect to the inductive limit topology.  Since $A(G)$ is dense in $\mathscr{C}(G)$ with respect to the inductive limit topology by Lemma~\ref{lem:ILTconv}, we can extend $\pi$ to a $*$-homomorphism $\tilde{\pi}: \mathscr{C}(G) \to B(\mathscr{H})$ that is also continuous with respect to the inductive limit topology.   Since $G$ is second countable, Renault's disintegration theorem \cite[Theorem~7.8]{MW08} gives that $\tilde{\pi}$ is bounded with respect to $\|\cdot\|_I$ and hence so is $\pi$.
\end{proof}

\section{\'Etale groupoids}
\label{sec:etaleautobound}

In this section, we will move away from ample groupoids and consider second countable, locally compact \'etale groupoids with Hausdorff unit spaces.    
By \cite[Proposition~3.10]{Exe08},  $\mathscr{C}(G)$ is the linear span of functions that are each in some $C_c(U)$ for an open bisection $U$ that is contained in a compact set. 
We show that every *-homomorphism from $\mathscr{C}(G)$ to $B(\mathscr{H})$ is bounded with respect to the I-norm.   
Previously this was only known for Hausdorff \'etale groupoids, see for example \cite[Lemma~2.3.3]{Sims18}.     
Note that we add the assumption that $G$ is second countable so that we have that $G$ is \emph{locally normal}.   
This also means we can work with sequences in $\mathscr{C}(G)$  instead of nets.  The main result of this section is the following.

\begin{thm}
\label{thm:etaleautobound}Let $G$ be a  second countable, locally compact \'etale groupoid with $\go$ Hausdorff.  
Suppose $\pi:\mathscr{C}(G) \to B(\mathscr{H})$ is a *-homomorphism for some Hilbert space $\mathscr{H}$.  Then $\pi$ is continuous with respect to the inductive limit topology on $\mathscr{C}(G)$.
\end{thm}

The proof of Theorem~\ref{thm:etaleautobound} will follow after we establish some lemmas that parallel the ideas of Lemma~\ref{lem:supportsinC} and Lemma~\ref{lem:boundedsummands}. 

\begin{lemma}
\label{lem:shrink}
Let $G$ be a  second-countable, locally compact \'etale groupoid with $\go$ Hausdorff.
Let $f \in \mathscr{C}(G)$ such that  $f = \sum_{V \in H}f_V$ where $H$ is a finite collection of open bisections and each $f_V \in C_c(V)$.  
Suppose $\supp^o(f) \subseteq O$ for some open set $O \subseteq G$.
Then there exists a finite collection of open bisections $F$ and
 functions $f_U \in C_c(U)$ for each $U \in F$  such that $\supp^o(f_U) \subseteq O$ and $f = \sum_{U \in F} f_U$.
\end{lemma}

\begin{proof}
We proceed by induction on the size of $H$.  For the base case, suppose $H$ has a single element $V$.  Then $f \in C_c(U)$  for some open bisection $U$ and
 by assumption $\supp^o(f) \subseteq O$ so $F=H$ suffices.
 
For the inductive step, suppose the lemma is true for functions that can be written as a sum of $n$ functions, each in some $C_c(V)$ for some $V$.  
Fix  $f \in \mathscr{C}(G)$ such that $\supp^o(f) \subseteq O$ and  $f=\sum_{V \in H}f_H$ with $|H|=n+1$.  

Fix $V \in H$. 
Let 
\[U_V = \supp^o(f_V) \cap O\]
 and for each $W \in H$, $W \neq V$, let
 \[ U_W = \supp^o(f_V) \cap \supp^o(f_W).\]
Since $\supp^o(f) \subseteq O$,  if $\gamma \in \supp^o(f_V) \setminus O$ then $\gamma \in U_W$ for some $W \in H$, $W \neq V$.  
So the collection $\{U_W\}_{W \in H}$ is an open cover of $\supp^o(f_V)$.

Since the open subset $\supp^o(f_V)$ is a normal topological space with respect to the subspace topology,  
we can find a partition of unity $\{\rho_W\}_{W \in H}$  in $C(\supp^o(f_V))$ (the continuous functions on $\supp^o(f_V)$) 
subordinate to $\{U_W\}_{W \in H}$ such that for each $\gamma \in \supp^o(f_V)$ we have
 $\sum_{W \in 
 H} \rho_W(\gamma)=1$ (see, for example, \cite[Theorem~5.1]{munkres}).
So $f_V = \sum_{W \in H} \rho_Wf_V$ (with pointwise multiplication). 
We claim that
\begin{enumerate}
\item \label{it1rho} For all $W \in H$, $\rho_Wf_V \in C_c(V)$,
\item \label{it2rho}for $W \neq V$, $\rho_Wf_V \in C_c(W)$.
\end{enumerate}
For item~\eqref{it1rho}, fix $W \in H$. Then $\rho_Wf_V $ is continuous on $\supp^o(f_V)$. We show it is continuous on all of $V$.  
Let $(\alpha_n)$ be a sequence in $V$ such that $\alpha_n \to \alpha \in V$. 
If $\alpha \in \supp^o(f_V)$, then $\alpha_n \in  \supp^o(f_V)$ eventually and 
$\rho_Wf_V(\alpha_n) \to \rho_Wf_V(\alpha)$.  
Suppose $\alpha \notin  \supp^o(f_V)$.  So $\rho_Wf_V(\alpha)=0$ and
\[| \rho_Wf_V(\alpha_{n})| = |\rho_W(\alpha_n)||f_V(\alpha_n)| \leq |f_V(\alpha_n)| \to 0.\]
Thus $\rho_Wf_V$ is continuous on $V$.  Since
$\overline{\supp^o(\rho_Wf_V)}^V \subseteq \overline{\supp^o(f_V)}^V$ is compact,  item~\eqref{it1rho} follows.

For item~\eqref{it2rho}, fix $W \in H$ with $W \neq V$. Then $\rho_Wf_V$ is continuous on $U_W$.
Let $(\alpha_n)$ be a sequence in $W$ such that $\alpha_n \to \alpha \in W$. 
We consider three cases.  For the first case,  suppose $\alpha \in \supp^o(f_V)$.   
Then $\alpha_n \in  \supp^o(f_V)$ eventually and 
$\rho_Wf_V(\alpha_n) \to \rho_Wf_V(\alpha)$.  
For the second case, suppose $\alpha \in  V \setminus \supp^o(f_V)$.
So $\rho_Wf_V(\alpha)=0$ and
\[| \rho_Wf_V(\alpha_{n})| = |\rho_W(\alpha_n)||f_V(\alpha_n)| \leq |f_V(\alpha_n)| \to 0.\]

For the third case, suppose $\alpha \in W \setminus V$. 
If $\alpha_n \notin \supp^o(f_V)$ eventually, then 
\[\rho_Wf_V(\alpha_n)=0=\rho_Wf_V(\alpha)\]
 eventually.  
So we may assume there is a subsequence such that every $\alpha_{n_j} \in \supp^o(f_V)$.
Since the closure of the  support of $f_V$ in $V$ is a compact normal subspace of $V$, it is complete and hence  there exists $\alpha_V \in V$ such that $\alpha_{n_j} \to \alpha_V$.  
Since $W$ is Hausdorff and $\alpha \in W$, $\alpha_V \notin W$ and hence $\rho_W(\alpha_V)=0$.   
Now  applying item~\eqref{it1rho} gives 
\[\rho_Wf_V(\alpha_{n_j}) \to \rho_Wf_V(\alpha_V) = \rho_W(\alpha_V)f_V(\alpha_V) = 0 = \rho_Wf_V(\alpha).\]
So $\rho_Wf_V$ is continuous on $W$.
We have $\overline{\supp^o(\rho_Wf_V)}^W \subseteq \overline{\supp^o(f_W)}^W$ which is compact proving item~\eqref{it2rho}. 

Now  
\[f=\rho_Vf_V + \sum_{W \in H, W \neq V} f_W + \rho_Wf_V\]
with $\supp^o(\rho_Vf_V) \subseteq O$.  To prove the lemma, it suffices to check that
\[\supp^o(\sum_{W \in H, W \neq V} f_W + \rho_Wf_V) \subseteq O\]
for then the inductive hypothesis applies.
By way of contradiction, suppose there exists $\gamma \in G \setminus O$ such that 
$\sum_{W \in H, W \neq V} f_W + \rho_Wf_V(\gamma) \neq 0$.
Then 
\[f(\gamma) = \rho_Vf_V(\gamma) + \sum_{W \in H, W \neq V} f_W + \rho_Wf_W(\gamma) = \sum_{W \in H, W \neq V} f_W + \rho_Wf_W(\gamma) \neq 0\]
which is a contradiction.
\end{proof}

\begin{lemma}
\label{lem:fix}
Let $G$ be a  second countable, locally compact \'etale groupoid with $\go$ Hausdorff.
Suppose $f \in \mathscr{C}(G)$ such that $\supp^o(f)  \subseteq K \subseteq  B$ for some compact set $K$ and open bisection $B$.  Then
$f \in C_c(B)$.
\end{lemma}

\begin{proof}
Using Lemma~\ref{lem:shrink}, write $f=\sum_{U \in F}f_U$ such that for each $U \in F$, 
 $f_U \in C_c(U)$ and $\supp^o(f_U) \subseteq B$.  We show for each $U \in F$ that $f_U \in C_c(B)$.
 Fix $U \in F$ and let $(\alpha_n) \subseteq B$ be a sequence such that $\alpha_n \to \alpha \in B$.  
 We show $f_U(\alpha_n) \to f_U(\alpha)$.
 Observe that 
 \[B = (\supp^o(f_U) \cap B) \cup  (B\setminus\supp^o(f_U))\]
so we consider two cases.
First suppose $\alpha$ is in the open set $\supp^o(f_U) \cap B$.  Then $\alpha_n$ is in $\supp^o(f_U) \cap B$ eventually
and $f_U(\alpha_n) \to f_U(\alpha)$ by the continuity of $f_U$ in $U$.

For the second case, suppose $\alpha \in B\setminus \supp^o(f_U)$.  Then $f_U(\alpha)=0$.  
We show $f_U(\alpha_n) \to 0$.   If $\alpha_n \in B\setminus \supp^o(f_U)$ eventually, then  $f_U(\alpha_n) =0$ eventually and we are done.  So suppose
 there exists a maximal subsequence such that every $\alpha_{n_j} \in \supp^o(f_U)$.  
 By maximal we mean that the $\alpha_n$ that do not appear in the subsequence are outside $\supp^o(f_U)$.
 Since the closure of the support of $f_U$ in $U$ 
 is a complete subspace of $U$,  there exists $\alpha_U \in U$ such that $\alpha_{n_j} \to \alpha_U$ and hence $f_U(\alpha_{n_j}) \to f_U(\alpha_U)$.
 If $\alpha_U=\alpha$, we are done.  Otherwise, since $B$ is Hausdorff and $\alpha \in B$,  $\alpha_U \notin B$. 
 So $f_U(\alpha_U) = 0$ and hence
$f_U(\alpha_{n_j})  \to 0$ which implies $f_U(\alpha_n) \to 0$ by maximality.

Thus $f \in C(B)$.  Notice $\overline{\supp^o(f)}^B \subseteq K$ and hence $f \in C_c(B)$.
\end{proof}

\begin{lemma}
\label{lem:supports}
Let $G$ be a  second countable, locally compact \'etale groupoid with $\go$ Hausdorff.
Let $\tilde{B}_1, ..., \tilde{B_k}$ and  $B_1, ..., B_k$ be open bisections and $K, K_1, ..., K_k$ are compact  sets such that for each  $1\leq i\leq k$ we have 
\[ \tilde{B}_i \subseteq K_i \subseteq {B}_i.\]
 Suppose $f \in \mathscr{C}(G)$ such that
$\supp^o(f) \subseteq K \subseteq \bigcup_{i=1}^k \tilde{B}_i$.  Then for each $1\leq i\leq k$, there exists
$f_i \in C_c(B_i)$ such that $f = \sum_{i=1}^k f_i$.
\end{lemma}

\begin{proof}
 We do a proof by induction on $k$ using the strategy inspired by the proof of \cite[Proposition~4.1]{Tu}.
The base case $k=1$ follows from Lemma~\ref{lem:fix}.
For the inductive step, suppose for every $h \in \mathscr{C}(G)$ with $\supp^o(h)$ inside some $K \subseteq \bigcup_{i=1}^k \tilde{B}_i$
with $k \geq 1$ satisfying the hypotheses of the lemma,  there exists  $h_i \in C_c(B_i)$  for each $1 \leq i \leq k$ such that  $h=\sum_{i=1}^k h_i$.  
Fix  $f \in \mathscr{C}(G)$ such that $f=\sum_{U \in H}f_U$ and  $\supp^o(f) \subseteq  K \subseteq \bigcup_{i=1}^{k+1} \tilde{B}_i$ as in the hypotheses of the lemma.  

Set \[
F =  \tilde{B}_{k+1} \setminus\bigcup_{i=1}^k \tilde{B_i}
.\]
Then $F$ is Hausdorff and  closed in $\tilde{B}_{k+1}$.
We claim that $f|_{F}$ is continuous on $F$.  To see this, it suffices to show each $f_U|_F$ is continuous on $F$.  Fix $U \in H$.
We can assume that $\supp^o(f_U) \subseteq  \bigcup_{i=1}^{k+1} \tilde{B}_i$ by Lemma~\ref{lem:shrink}.
Let $(\alpha_n) \subseteq F$ such that $\alpha_n \to \alpha \in F$. 
 If $\alpha \in \supp^o(f_U)$ which is open, 
 then $\alpha_n \in \supp^o(f_U)$ eventually and we get $f_U|_F(\alpha_n) \to f_U|_F(\alpha)$ by continuity in $U$.  
 So assume $\alpha \notin \supp^o(f_U)$.  We consider two cases:  First suppose 
$\alpha_n \in F \setminus \supp^o(f_U)$ eventually.  
Then $f_U|_F(\alpha_n) =0= f_U|_F(\alpha)$ eventually. 
For the second cases, suppose there exists a maximal subsequence such that  for every $n_j$, 
 $\alpha_{n_j} \in \supp^o(f_U) \cap F$.  Then if  $\alpha_n$ is not included in the subsequence,  $f_U|_F(\alpha_n)=0$.
 Since the closure of $\supp^o(f_U)$ is a complete subspace of $U$,  there exists $\alpha_U \in U$ such that  $\alpha_{n_j} \to \alpha_U$
 and $f_U(\alpha_{n_j}) \to f_U(\alpha_U)$.  If $\alpha_U \in F$, then $\alpha_U = \alpha$ because $F$ is Hausdorff.  Finally, suppose $\alpha_U \notin F$.  
 We claim that $\alpha_U \notin \supp^o(f_U)$ which suffice for then 
 $f_U|_F(\alpha_{n}) \to  f_U|_F(\alpha)$.  
 By way of contradiction, suppose $\alpha_U \in \supp^o(f_U)$.  
 So there exists $\tilde{B}_i,$  with $i\neq k+1$ such that $\alpha_U \in \tilde{B}_i$. Since $\tilde{B}_i$ is open,  there exists $\alpha_{n_j}$ such that $\alpha_{n_j} \in \tilde{B}_i$. But by assumption every $a_{n_j} \in F$ and $F \cap \tilde{B_i} = \emptyset$,  which is a contradiction.  
 Thus we have shown $f_U|_F$ is continuous on $F$ and hence so is $f|_F$.

Notice $f|_F$ is 0 outside 
\[K' = K  \setminus\bigcup_{i=1}^k \tilde{B_i} \]
which is closed in $K$ and hence compact.  Since $K' \subseteq F$,  we have $f|_F \in C_c(F)$.
Now we apply Tietze's extension theorem to get a function $f_{k+1} \in C_c(\tilde{B}_{k+1}) \subseteq C_c(B_{k+1})$ such that $f_{k+1}|_F = f|_F$. 
Now $f=(f-f_{k+1}) + f_{k+1}$.   Notice that $f-f_{k+1} \in \mathscr{C}(G)$ and is 0 on $F$.  In fact 
\[
\supp^o(f-f_{k+1}) \subseteq \bigcup_{i=1}^k \tilde{B_i}.
\] 
Now apply the inductive hypothesis to $f-f_{k+1}$ to get the remaining $f_i$'s.
\end{proof}

\begin{lemma}
\label{lem:etalebdsummands}
Let $G$ be a  second countable, locally compact \'etale groupoid with $\go$ Hausdorff.  
Let $B_1, \dots, B_k$ be open bisections.
Suppose $f \in \mathscr{C}(G)$ such that $f= \sum_{i=1}^k f^0_i$ where each $f_i^0 \in C_c(B_i)$.
Let $\epsilon > 0$.
Then for each $i$ there exists  $f_{i} \in C_c(B_i)$ such that 
\begin{enumerate}
\item \label{it1:etalebdsummand}$\|f_{i}\|_{\infty}\leq 2\|f\|_{\infty}+\epsilon$ and 
\item\label{it2:etalebdsummand}$f = \sum_{i=1}^k f_{i}$.
\end{enumerate}
\end{lemma}

\begin{proof}
In what follows, we adjust these functions so that  \eqref{it2:boundedsummand} still holds 
in order to ensure  \eqref{it1:etalebdsummand} holds as well.  
Fix $\epsilon>0$ and fix $1 \leq i \leq k$ such that $\|f^0_i\|_{\infty}>2 \|f\|_{\infty}+\epsilon$.   If no such $i$ exists, we are finished. 
Set 
\[O = \{\gamma \in B_i : |f^0_{i}(\gamma)| < \|f\|_{\infty} + \frac{\epsilon}{2}\}.\]
Then $O$ is open.
Since  $\|f^0_i\|_{\infty}> \dfrac{\|f^0_i\|_{\infty}}{2} > \|f\|_{\infty}+\frac{\epsilon}{2}$, the set $B_i \setminus O$ is nonempty. 
Further $B_i\setminus O$ is closed in $B_i$ and contained in the compact support of $f_i^0$  (also inside $B_i$) and hence $B_i\setminus O$ is compact.  

Fix $\gamma \in B_i \setminus O$.  
Then $|f^0_i(\gamma)| > \|f\|_{\infty}$ so in particular $f(\gamma) \neq f^0_i(\gamma)$ and hence there exists a maximal nonempty 
\[S_{\gamma}\subseteq \{1, ..., i-1, i+1, ..., k\}\] 
such that 
\[\gamma \in B_i \cap \left(\bigcap_{j \in S_{\gamma}}B_j\right)\quad \text{ and } \quad 
f(\gamma) = {f}_{i}^0(\gamma) + \sum_{j \in S_{\gamma}} {f}_{j}^0(\gamma).\]

Choose an open bisection 
$C_{\gamma, S_{\gamma}} \subseteq B_i \cap \left( \bigcap_{j \in S_{\gamma}}B_j \right)$ such that 
\begin{itemize}
\item $\gamma \in C_{\gamma,s_{\gamma}}$,
\item the restriction of $ ({f}_{i}^0 + \sum_{j \in S_{\gamma}} {f}_{j}^0)$ to $C_{\gamma, S_{\gamma}}$ is continuous and
\item  for each 
$\alpha \in C_{\gamma, S_{\gamma}}$ we have 
\begin{equation}
\label{eqn:etaleC}|({f}_{i}^0(\alpha) + \sum_{j \in S_{\gamma}} {f}_{j}^0(\alpha)) - f(\gamma)| \leq \frac{\epsilon}{2}.
\end{equation}
\end{itemize}
Then the collection $\{C_{\gamma, S_\gamma}\}_{\gamma \in B_i \setminus O}$ covers $B_i \setminus O$ so there exists a finite subcover $\{C_{\gamma_p, S_{p}}\}_{p=1}^s$.   

Now, we adjust the functions.    For notational convenience, 
let $C_{\gamma_0, S_{0}} =O$.  
Choose a partition of unity $\{\rho_p\}_{p=0}^s$ in $C(B_i)$ subordinate to the cover 
$\{C_{\gamma_p, S_p}\}_{p=0}^s$.   Define
\[
f_{i}={f}_{i}^0 + \sum_{j=1}^k \  \sum_{\{p: j \in S_p\}} \rho_p{f}_{j}^0 \ \text {  (with pointwise multiplication)}.\]
To ensure \eqref{it2:etalebdsummand} still holds,  for $j \neq i$ define 
 \[
f_{j} = {f}_{j}^0 - \sum_{\{p: j \in S_p\}} \rho_p{f}_{j}^0 \text {  (with pointwise multiplication)}.\]
Since the compactly supported functions are an ideal in the ring of continuous functions, 
we have $f_i \in C_c(B_i)$  and for each $j$,  $f_j \in C_c(B_j)$.

Notice for $j \neq i$ we have  $\|f_{j} \|_{\infty} \leq \| {f}_{j}^0\|_{\infty}$.
We claim that $\|f_{i}\|_{\infty} \leq 2\|f\|_{\infty}+\epsilon$.  
To prove the claim, fix $\alpha \in G$.  It suffices to show $|f_{i}(\alpha)| \leq 2\|f\|_{\infty}+\epsilon$.   
We compute
\begin{align*}
|f_{i}(\alpha)| &=|{f}_{i}^0(\alpha) + \sum_{j=1}^k \  \sum_{\{p: j \in S_p\}} \rho_p(\alpha){f}_{j}^0(\alpha)| \\
 &=|\sum_{p=0}^s\rho_p(\alpha){f}_{i}^0(\alpha) + \sum_{p=1}^s \  \sum_{j \in S_p} \rho_p(\alpha){f}_{j}^0(\alpha)| \\
 &=|\rho_0(\alpha){f}_{i}^0(\alpha) + \sum_{p=1}^s \ \rho_p(\alpha)({f}_{i}^0(\alpha)+ \sum_{j \in S_p} {f}_{j}^0(\alpha))| \\
 &\leq(\|f\|_{\infty}+\frac{\epsilon}{2})  + \sum_{p=1}^s \ \rho_p(\alpha)|{f}_{i}^0(\alpha)+ \sum_{j \in S_p} {f}_{j}^0(\alpha)| \\
 &\leq(\|f\|_{\infty}+\frac{\epsilon}{2})  + \sum_{p=1}^s \ \rho_p(\alpha)( |f(\gamma_p)| +\frac{\epsilon}{2}) \text{ by \eqref{eqn:etaleC}}\\
 &\leq(\|f\|_{\infty}+\frac{\epsilon}{2})  + \sum_{p=1}^s \ \rho_p(\alpha)( \|f\|_{\infty} +\frac{\epsilon}{2})\\
&\leq 2\|f\|_{\infty}+\epsilon.
\end{align*}

Now check if there are any $j$ such that $\|f_{j}\|_{\infty} >2 \|f\|_{\infty}+\epsilon$.  
If so, pick one such $j$ and repeat the above process to get a new collection of functions continuing 
until no such $j$ exists and hence \eqref{it1:etalebdsummand} holds. 
\end{proof}

\begin{proof}[Proof of Theorem~\ref{thm:etaleautobound}.]
Fix a $*$-homomorphism $\pi:\mathscr{C}(B)(G) \to B(\mathscr{H})$. 
Suppose $f_n \to f$ uniformly and $\supp^o(f_n) \subseteq K$ eventually for some compact $K$.  
 We claim that there are open bisections $\tilde{B}_1, ..., \tilde{B}_k$ and $B_1, ..., B_k$ and compact sets
 $K_1, ... K_k$ such that for each $i, 1 \leq i \leq k$ we have
 \[\tilde{B}_i \subseteq K_i \subseteq B_i \quad
 \text{ and } \quad K \subseteq \bigcup_{i=1}^k \tilde{B}_i.\]
To see this, for each $\gamma \in K$, let $B_{\gamma}$ be an open bisection containing $\gamma$.  
By local compactness, $\gamma$ has a neighborhood base of compact sets, so we can find $K_{\gamma}$ and $\tilde{B}_{\gamma}$ such that 
\[\gamma \in \tilde{B}_{\gamma} \subseteq K_\gamma \subseteq {B}_{\gamma}.\]
Then the $\tilde{B}_{\gamma}$'s cover $K$ and we can find a finite subcover as needed.  

Fix $\epsilon >0$.  Since $\supp^o(f_n-f) \subseteq K$ eventually, we can apply
Lemma~\ref{lem:supports} and then Lemma~\ref{lem:etalebdsummands} to eventually write each $f_n-f = \sum_{i=1}^k f_{n,i}$  such that each $f_{n,i} \in C_c(B_i)$ with  $\supp^o(f_{n,i}) \subseteq B_i$ and   
\[\|f_{n,i}\|_{\infty}\leq 2\|f_n-f \|_{\infty}+ \dfrac{\epsilon}{2k}.\]  
Choose $N$ so that if $n \geq N$, then 
\[ \|f_n-f\|_{\infty} \leq \dfrac{\epsilon}{4k}.\]
Now we compute for $n \geq N$, applying \cite[Proposition~3.14]{Exe08} at the second inequality ,
\[\|\pi(f_n-f)\|\leq \sum_{i=1}^k\|\pi( f_{n,i})\| \leq \sum_{i=1}^k \|f_{n,i}\|_{\infty} \leq k (\|f_n-f\|_{\infty}+\dfrac{\epsilon}{2k}) \leq \epsilon.  \]
\end{proof}

The following corollary follows immediately from Theorem~\ref{thm:etaleautobound} and the disintegration theorem \cite[Theorem~7.8]{MW08}.

\begin{cor}
\label{cor:etaleautobound}Let $G$ be a  second countable, locally compact \'etale groupoid with $\go$ Hausdorff.  
Suppose $\pi:\mathscr{C}(G) \to B(\mathscr{H})$ is a *-homomorphism for some Hilbert space $\mathscr{H}$.  Then $\pi$ is bounded with respect to the $I$-norm on $\mathscr{C}(G)$.
\end{cor}

\section{Reconciling our construction and those done previously}
\label{sec:same}

\noindent{\textbf{The reduced C*-algebra}}  The restriction of the left regular representation $\pi$ of $\mathscr{C}(G)$
to $A(G)$ is precisely the left regular representation of $A(G)$.   So the reduced C*-algebra we have 
constructed $\overline{\pi(A(G))}$ is contained in the usual one $\overline{\pi(\mathscr{C}(G))}$.
It suffices to show $\pi(A(G))$ is dense in $\pi(\mathscr{C}(G))$. 
Fix $\pi(f) \in \pi(\mathscr{C}(G))$.  Lemma~\ref{lem:ILTconv} says $A(G)$ is dense in $\mathscr{C}(G)$ in the inductive limit topology so there exists a sequence $(f_n) \subseteq A(G)$ such that $f_n \to f$ in the inductive limit topology.  
From Proposition~\ref{prop:autobound},  $\pi$ is continuous with respect to the inductive limit topology so $\pi(f_n) \to \pi(f)$ as needed. 

\medskip

\noindent{\textbf{The full C*-algebra}} 
\begin{thm} Let $G$ be a locally compact \'etale groupoid with Hausdorff unit space.
\label{thm:recovery}
Consider the following 4 norms where the first is on $A(G)$ and the other three are on $\mathscr{C}(G)$: 
\begin{align*}
\|f\|_1 &= \sup \{\|\pi(f)\| : \pi:A(G) \to B(\mathscr{H})\text{ is a *-hm for some } \mathscr{H}\} \text{ (Theorem~\ref{thm:fullnorm})}\\
\|f\|_2 &=\sup \{\|\pi(f)\| : \pi:\mathscr{C}(G) \to B(\mathscr{H})\text{ is a *-hm for some } \mathscr{H}\}\text{ (for example, see \cite[Def.~3.17]{Exe08}})\\
\|f\|_3 &= \sup \{\|\pi(f)\| : \pi:\mathscr{C}(G) \to B(\mathscr{H})\text{ is a *-hm that is continuous with respect to the}\\
& \hspace{1.2cm} \text{ inductive limit topology for some } \mathscr{H}\}\text{ (for example, see \cite[Def.~II.1.3]{Ren80}})\\
\|f\|_4 &= \sup \{\|\pi(f)\| : \pi:\mathscr{C}(G) \to B(\mathscr{H})\text{ is an I-norm bounded *-hm for some } \mathscr{H}\}\\
& \hspace{1.2cm} \text{(for example, see \cite[Page~101]{Pat99}})
\end{align*}
\begin{enumerate}
\item\label{it1:recovery} If $G$ is ample, then our C*-algebra  $\overline{A(G)}^{\|\cdot\|_1}$ (see Def.~\ref{def:main}) is the same as $\overline{\mathscr{C}(G)}^{\|\cdot\|_i}$ for $i=2,3$.  
\item\label{it2:recovery} If $G$ is a second countable, then $\overline{\mathscr{C}(G)}^{\|\cdot\|_i}$ for $i=2,3,4$ are all the same.  
\end{enumerate}
\end{thm}

\begin{remark}
Typically $\|\cdot\|_2$ only appears in the Hausdorff setting where the existence of the  ``inductive limit topology'' on $C_c(G)$ is guaranteed, see for example, \cite[Proposition~D.1]{tfb}.   When $G$ is not Hausdorff, this definition of $\|\cdot\|_2$ still makes sense but is less helpful.  \end{remark}

\begin{proof}
For item~\ref{it1:recovery}, suppose $G$ is ample. 
Note that every $*$-homomorphism $\pi:\mathscr{C}(G) \to B(\mathcal{H})$ restricts to a $*$-homomorphism on $A(G)$.  By Proposition~\ref{prop:autobound} and  Lemma~\ref{lem:ILTconv}, any $*$-homomorphism on $A(G)$ extends to a $*$-homomorphism on $\mathscr{C}(G)$ that is continuous with respect to the inductive limit topology.  Putting this all together gives
$\|f\|_1 = \|f\|_2=\|f\|_3$ for $f \in A(G)$.
 Thus we have  \[\overline{A(G)}^{\|\cdot\|_1}=\overline{A(G)}^{\|\cdot\|_2}= \overline{A(G)}^{\|\cdot\|_3}\subseteq \overline{\mathscr{C}(G)}^{\|\cdot\|_i}\]
for $i=2,3$.  Lemma~\ref{lem:ILTconv} also implies $\mathscr{C}(G) \subseteq \overline{A(G)}^{\|\cdot\|_3}$ and item~\ref{it1:recovery} follows.

For item~\ref{it2:recovery}, when $G$ is second countable, Theorem~\ref{thm:etaleautobound} and Corollary~\ref{cor:etaleautobound} imply all the norms on $\mathscr{C}(G)$ are the same.     
\end{proof}

\bibliography{bibliography}
\bibliographystyle{plain}
 
\end{document}